\documentclass[a4paper,12pt]{article}
\usepackage{amsthm}
\usepackage{amsfonts}
\usepackage{amsmath}
\usepackage{amssymb}
\usepackage[pdftex]{graphicx}
\usepackage{hyperref}
\usepackage{datetime}
\usepackage{commath}
\usepackage{float}
\restylefloat{table}
\newtheorem{theorem}{Theorem}[section]
\theoremstyle{definition}

\theoremstyle{remark}

\theoremstyle{plain}
\newtheorem{lemma}[theorem]{Lemma}

\newtheorem{proposition}[theorem]{Proposition}
\newcommand{\N}{\mathbb{N}}

\newcommand{\suminfty}{\sum_{\N}{a_n}}

\begin{document}
\title{A series of series topologies on $\mathbb{N}$}
\author{Jason DeVito \thanks{The University of Tennessee at Martin email: jdevito@utm.edu} \and Zachary Parker \thanks{The University of Tennessee at Martin; email: zacppark@ut.utm.edu}}
\date{}

\maketitle

\begin{abstract}  Each series $\sum_{n=1}^\infty a_n$ of real strictly positive terms gives rise to a topology on $\mathbb{N} = \{1,2,3,...\}$ by declaring a proper subset $A\subseteq \mathbb{N}$ to be closed if $\sum_{n\in A} a_n < \infty$.  We explore the relationship between analytic properties of the series and topological properties on $\mathbb{N}$.  In particular, we show that, up to homeomorphism, $|\mathbb{R}|$-many topologies are generated.  We also find an uncountable family of examples $\{\N_\alpha\}_{\alpha \in [0,1]}$ with the property that for any $\alpha < \beta$, there is a continuous bijection $\N_\beta\rightarrow \N_\alpha$, but the only continuous functions $\N_\alpha\rightarrow \N_\beta$ are constant.

\end{abstract} 

\section{Introduction}\label{sec:intro}
Consider a series $\sum_{n=1}^\infty a_n$ whose terms are strictly positive real numbers.  For any $A\subseteq\mathbb{N} = \{1,2,3,...\}$, we will use the notation $\sum_A a_n$ as a shorthand for $\sum_{n\in A} a_n$.

If $A_1,A_2\subseteq \mathbb{N}$ are chosen so that $\sum_{A_1} a_n < \infty$ and $\sum_{A_2} a_n <\infty$, then $$\sum_{A_1 \cup A_2} a_n \leq \sum_{A_1}a_n + \sum_{A_2} a_n < \infty.$$  Further, if $\sum_{A} a_n < \infty$ and $B\subseteq A$, then $\sum_B a_n \leq \sum_A a_n< \infty$.  Thus, the collection $\{A: \sum_A a_n < \infty\} \cup \{\N\}$ forms a topology of closed sets on $\mathbb{N}$.  We use the notation $\N_{a_n}$ to refer to $\N$ equipped with this topology.

If one does not include $\mathbb{N}$ in the topology, then one obtains an ideal of sets $\mathcal{I}\subseteq \mathcal{P}(\mathbb{N})$, called a summable ideal.  Summable ideals have been studied considerably by set theorists (see, for example \cite{Fl,Hr,Ka,KwTr,Mr}), but do not seem to have been considered from the topological perspective before.

One focus of our paper is to relate analytic properties of the series with some the usual basic topological properties.  For example, we can characterize when $\N_{a_n}$ is connected.

\begin{theorem}\label{thm:conv}  The series $\sum_\N a_n$ diverges iff $\N_{a_n}$ is connected.

\end{theorem}

Likewise, we have a characterization of compactness.

\begin{theorem}\label{thm:comp} The series $\sum_\N a_n$ has $\inf\{a_n\} > 0$ iff $\N_{a_n}$ is compact.

\end{theorem}

Another focus of this paper is on continuity of maps to and from an $\N_{a_n}$.  We first establish that continuous maps between an $\N_{a_n}$ and a ``nice'' topological space are constant, unless $\N_{a_n}$ is discrete.  More specifically, we prove the following theorem.

\begin{theorem}\label{thm:nicecont}
If $X$ is compact, connected, and Hausdorff, or $X$ is path-connected then any continuous map $f:X\rightarrow \N_{a_n}$ is constant.  If $Y$ is metrizable and $\sum_\N a_n$ diverges, then any continuous function $f:\N_{a_n}\rightarrow Y$ is constant.
\end{theorem}

We also investigate continuous maps $\N_{a_n}\rightarrow \N_{b_n}$.  Our main theorem finds sufficient conditions for $\N_{a_n}$ and $\N_{b_n}$ to have distinct homeomorphism types.

\begin{theorem}\label{thm:homeo}
Suppose $\sum_\N a_n$ and $\sum_\N b_n$ are both series consisting of positive terms and assume $\lim_{n\rightarrow \infty} a_n =\lim_{n\rightarrow\infty} \frac{a_n}{b_n} = 0$ and that $\sum_\N b_n $ diverges.  Then $\N_{a_n}$ and $\N_{b_n}$ are not homeomorphic.

\end{theorem}

By considering $a_n = \frac{1}{n^p}$ and $b_n = \frac{1}{n^q}$ for $0\leq p < q \leq 1$, we show, as a simple corollary, that this theorem gives $|\mathbb{R}|$-many homeomorphism types among the $\N_{a_n}$

In fact, Theorem \ref{thm:homeo} can be strengthened in the case where $a_n = \frac{1}{n^p}$ and $b_n = \frac{1}{n^q}$ for $p\in[0,1], p<q$.

\begin{theorem}\label{thm:constant}  Suppose $0\leq p \leq 1$ and that $p < q$.  Then the only continuous functions $\N_{1/n^p}\rightarrow \N_{1/n^q}$ are constant.
\end{theorem}

Often when initially learning topology, one questions whether a continuous bijection is necessarily a homeomorphism.  Of course, there are simple counterexamples to this conjecture, e.g., $e^{i\theta}:[0,2\pi)\rightarrow S^1$, but Theorem \ref{thm:constant} provides an interesting uncountable family of counterexamples:  the identity function $\N_{1/n^q}\rightarrow \N_{1/n^p}$ is a continuous bijection, but the only continuous functions $\N_{1/n^p}\rightarrow \N_{1/n^q}$ are constants.

We also find an interesting uncountable family of spaces with diversity one.  Recall that a topological space is said to have \textit{diversity one} or to have \textit{homeomorphic open sets} if any two-nonempty open sets are homeomorphic.  Such topological spaces have been previously studied \cite{FrRa}, under the added restriction that the underlying topology is Hausdorff.

\begin{theorem}\label{thm:hos}
Suppose $\inf \{a_n\} = 0$ and $\sum_\N a_n$ diverges.  Then every non-empty open subset of $\N_{a_n}$ is homeomorphic to $\N_{a_n}$.

\end{theorem}

Theorems \ref{thm:homeo}, \ref{thm:constant}, and \ref{thm:hos} can be deduced from more general results in \cite{KwTr} and \cite{Mr}.  We choose to include proofs of these theorems for two reasons.  First, our proofs use simpler tools.  In fact, they employing nothing beyond the curriculum of a standard American calculus II course.  Second, the papers \cite{KwTr} and \cite{Mr} are intended for set theorist.  While the intersection of topologists and set theorists is large, we hope that by expressing the results in topological language, they become more widely accessible.

The outline of this paper is as follows.  In Section \ref{sec:basic}, we establish the connection between analytic properties of the series and basic topological properties (compactness, connectedness, etc), proving Theorems \ref{thm:conv} and \ref{thm:comp}.  In Section \ref{sec:cont}, we establish properties about continuous maps between ``nice spaces'' and $\N_{a_n}$, proving Theorem \ref{thm:nicecont}.  Finally, in Section \ref{sec:four}, we study continuous functions $\N_{a_n}\rightarrow \N_{b_n}$, proving Theorems \ref{thm:homeo}, \ref{thm:constant}, and \ref{thm:hos}.

We would like to acknowledge support from the Bill and Roberta Blankenship Undergraduate Research Endowment. We would also like to thank Alexey Lebedev and Georgiy Shevchenko for several proofs and ideas appearing in Section \ref{sec:four}.  Finally, we would like to thank an anonymous referee who informed us that the main theorems in Section \ref{sec:four} are special cases of theorems found in \cite{KwTr} and \cite{Mr}.

\section{Basic topological properties}\label{sec:basic}

As mentioned in the introduction, given any series of positive terms $\sum_\N a_n$, one can construct a topology on $\N$ by declaring a subset $A\subseteq \N$ to be closed iff $\sum_A a_n < \infty$ or $A = \N$.  The restriction that all $a_n$ be positive serves two purposes.  First, it means that the convergence of a series of the form $\sum_{A} a_n$ for $A\subseteq \N$ is independent of the order we take the sum.  Second, as the following example shows, allowing a mixture of infinitely many positive and negative terms can lead to cases where the above prescription does not actually define a topology.

\begin{proposition}  The set $\left\{A\subseteq \N: \sum_A \frac{(-1)^{n+1}}{n} < \infty\right\}\cup\{\N\}$ is not a topology of closed sets on $\N$.

\end{proposition}

\begin{proof}
For each $n\in \N$, let $A_n$ denote the set $\N\setminus\{2, 4, 6, ..., 2n\}$.  Then $\sum_{A_n} a_n < \infty$, because $A_n$ differs from $\N$ only on a finite set and $\sum_\N a_n = \ln(2) < \infty$.

On the other hand $\bigcap_{n\in \N} A_n$ is precisely the odd numbers $O$, and $$\sum_{O} a_n = \sum_\N \frac{1}{2n+1} = \infty.$$

\end{proof}

Of course, the above proof can be adapted to any conditionally convergent series.  Henceforth, all the terms in any sum will be positive real numbers.

We begin with some basic topological properties of $\N_{a_n}$.

\begin{proposition}\label{prop:basicproperties}
For any series of positive terms $\suminfty$, the topological space  $\N_{a_n}$ has the following properties.

\begin{enumerate}  \item  Points are closed.

\item If $B\subseteq A\subsetneq\N$ with $A$ closed, then $B$ is closed.

\item  For any $A\subseteq \N$, either $\overline{A} = A$ or $\overline{A} = \N$.

\end{enumerate}
\end{proposition}

\begin{proof}  First, $(1)$ follows because every finite sum automatically converges.

For $(2)$, let $B\subseteq A\subsetneq \N$ with $A$ closed.  Then $\sum_B a_n \leq \sum_A a_n < \infty$, so $B$ is closed as well.

For $(3)$, if $\overline{A} \neq \N$, then $\sum_A a_n \leq \sum_{\overline{A}} a_n < \infty$, so $A$ is closed, that is, $A = \overline{A}$.

\end{proof}

In fact, the second and third properties in Proposition \ref{prop:basicproperties} are equivalent in any topological space, as shown in the next proposition.

\begin{proposition}  Suppose $X$ is a topological space.  Then $X$ has the property that every subset is closed or dense iff $X$ has the property that every subset of a proper closed set is closed.

\end{proposition}

\begin{proof}
Assume initially that the closure of every set is itself or $X$.  Let $A\subsetneq X$ be a closed set and let $B\subseteq A$ be arbitrary.  Then, $\overline{B}\subseteq A$, so $\overline{B}\neq X$.  By assumption, $\overline{B} = B$ so $B$ is closed.

Conversely, assume every subset of a proper closed set is closed and let $A\subseteq X$ be arbitrary.  If $\overline{A}\neq X$, then $\overline{A}$ is a proper closed set.  Since $A\subseteq \overline{A}$, $A$ must be closed.

\end{proof}

By taking complements, it follows immediately that supersets of non-empty open sets are open, and, equivalently, that the interior of any subset is either itself or empty.

\

We now show that many familiar topological properties are equivalent to natural analytic properties.  We begin with a characterization of the discrete topology, which in particular, proves Theorem \ref{thm:conv}.

\begin{theorem}\label{thm:disc} The following are equivalent.

\begin{enumerate}
\item $\sum_\N a_n$ converges.

\item $\N_{a_n}$ is discrete.

\item  $\N_{a_n}$ is disconnected.

\item  $\N_{a_n}$ is Hausdorff

\end{enumerate}

\end{theorem}

\begin{proof}
To show $(1)\implies (2),$  Suppose $\sum_\N a_n$ converges and let $A\subseteq \N$ be arbitrary.  Then $\sum_A a_n \leq \sum_\N a_n < \infty$, so every set is closed.  This shows $\N_{a_n}$ is discrete.

$(2)\implies (4)$ is obvious, as is $(2)\implies (3)$ for any set with more than one point.

To see that $(3)\implies (1)$ and $(4)\implies(1)$, we note that both conditions $(3)$ and $(4)$ imply there are two non-empty proper closed $A_1, A_2$ for which $A_1\cup A_2 = \N$.  For (3), one can obtain $A_1$ and $A_2$ as complements of any two disconnecting open sets.  For $(4)$, one can obtain $A_1$ and $A_2$ as complements of any two disjoint proper open sets.  Because both $A_i$ are proper, $\sum_{A_i} a_n < \infty$.  Thus, $\sum_{\N} a_n \leq \sum_{A_1} a_n + \sum_{A_2} a_n < \infty$, so the series converges.

\end{proof}

We now give a similar characterization of when the topology on $\N$ is the cofinite topology, which will encompass Theorem \ref{thm:comp}.

\begin{theorem}\label{thm:cof}  The following are equivalent.

\begin{enumerate}
\item  $\inf \{a_n\} > 0$

\item  $\N_{a_n}$ is the cofinite topology.

\item  $\N_{a_n}$ is compact.
\end{enumerate}

\end{theorem}

Before proving this, we need a lemma which will be used again in Section \ref{sec:four}.

\begin{lemma}\label{lem:closedsub}Suppose $\inf\{ a_n\} = 0$.  Then there is an infinite closed set.

\end{lemma}

\begin{proof} Because $\inf\{a_n\} = 0$, there is an $n_1\in \N$ with $a_{n_1} \leq \frac{1}{1^2}$.  Continuing inductively, there is an $n_k\in \N$ for which both $n_k > n_{k-1}$ and $a_{n_k} \leq \frac{1}{k^2}$.  Setting $A = \{n_1, ... ,n_k, ...\}$, we see that $\sum_A a_n \leq \sum_\N \frac{1}{k^2} < \infty$.  Thus, $A\subseteq \N_{a_n}$ is closed.

\end{proof}

We now prove Theorem \ref{thm:cof}.  
\begin{proof}
For $(1)\implies (2)$, let $L = \inf\{a_n\}$ and assume $L > 0$.  If $A\subseteq \N$ is infinite, then $\sum_A a_n \geq \sum_A L = \infty$, so $A$ is not closed, unless $A = \N$.  Since we have previously shown that finite sets are closed, this is precisely the cofinite topology.

For $(2)\implies (3)$, we note that any non-empty open set can only miss finitely many points.  It follows easily that the cofinite topology on any set is compact.

Finally, we show $(3)\implies (1)$, via the contrapositive.  So, assume $\inf \{a_n\}$ is not greater than $0$.  Since the terms $a_n$ are all positive, this implies $\inf \{a_n\} = 0$.  By Lemma \ref{lem:closedsub}, there is an infinite closed subset $A = \{n_1,n_2,...n_k,...\}$.

Now, for each $i\in \N$, we let $U_i = A^c \cup\{n_1,..., n_i\}$.  We claim $\{U_i\}$ forms an open cover of $\N_{a_n}$ with no finite subcover.  To see that each $U_i$ is open, simply note that $U_i^c = A\setminus\{n_1,...,n_i\}$ is a subset of the closed set $A$, so is closed.  Further, the $U_i$ cover $\N$ because for each $n_i\in A$, $n_i\in U_i$ and for $n\notin A$, $n\in U_1$.

On the other hand, if $\{U_{i_1}, ... U_{i_k}\}$ is a finite collection of the $U_i$ and $l > \max\{i_1,..., i_k\}$, then $n_{l} \notin U_{i_1}\cup....\cup U_{i_k}$.  So, there is no finite subcover.

\end{proof}

\section{Continuous functions to and from ``nice'' spaces}\label{sec:cont}

In this section, we study continuous functions between ``nice'' spaces and an $\N_{a_n}$, proving Theorem \ref{thm:nicecont}.  In fact, Propositions \ref{prop:XtoN} and \ref{prop:NtoX} together are equivalent to Theorem \ref{thm:nicecont}.

We begin this section with a characterization of continuous functions $f:X\rightarrow \mathbb{N}_{a_n}$ where $X$ is a continuum or $X$ is path-connected.  We recall that a topological space is called a \textit{continuum} if it is compact, connected, and Hausdorff.

\begin{proposition}\label{prop:XtoN}Suppose $\sum_\N a_n$ is a series of positive terms.  If $X$ is a continuum or path-connected, then, any continuous function $f:X\rightarrow \N_{a_n}$ is constant.

\end{proposition}

In fact, the result holds more generally (with the same proof) if $\mathbb{N}_{a_n}$ is replaced with any countable topological space for which points are closed.

\begin{proof}  Assume initially that $X$ is a continuum.  As discussed previously, in any series topology, finite sets are closed.  Now, since $f$ is continuous, for each $n\in \mathbb{N}$, $f^{-1}(n)$ is a closed subset of $X$.  Thus, we may write $X$ as a countable disjoint union of closed sets: $X = \coprod_{n\in\mathbb{N}} f^{-1}(n)$.  According to Sierpinski \cite{Sier}, this is only possible if at most one of the closed sets is non-empty.  That is, $f$ is constant.

Now, assume $X$ is path connected and $f:X\rightarrow \N_{a_n}$ is continuous.  Let $p,q\in X$ be arbitrary and let $\gamma:[0,1]\rightarrow X$ be a curve with $\gamma(0) = p$ and $\gamma(1) = q$.  Then, $f\circ \gamma$ is continuous and $[0,1]$ is a continuum, $f\circ \gamma$ must be constant.  Thus, $f(p) = f(q)$ as claimed.  It follows that $f$ is constant.

\end{proof}

In particular, each $\N_{a_n}$ is totally path disconnected.  This contrasts with Theorem \ref{thm:conv}, which asserts that $\N_{a_n}$ is connected if $\sum_\N a_n$ diverges.

We now change focus to the case where the domain is an $\N_{a_n}$.  If $\sum_\N a_n$ converges, then the induced topology is discrete (Theorem \ref{thm:disc}), so any function $f:\mathbb{N}\rightarrow Y$ is continuous for any topological space $Y$.  In particular, there are many non-constant continuous functions $f:\N_{a_n}\rightarrow Y$.  Nonetheless, we now show this behavior is limited to convergent series.

\begin{proposition}\label{prop:NtoX}
Suppose $\mathbb{N}_{a_n}$ is associated to a divergent series $\sum_\N a_n$.  If $X$ is metrizable, then any continuous function $f:\mathbb{N}\rightarrow X$ is constant.
\end{proposition}

\begin{proof}

By Theorem \ref{thm:conv}, the topology on $\mathbb{N}$ is connected, so $f(\mathbb{N})$ is a connected subset of $X$.  We recall that a connected set in a metrizable space is either a single point, or has cardinality at least that of $\mathbb{R}$.  Indeed, suppose $C$ is connected with $1 < |C|<|\mathbb{R}|$, and let $c_1,c_2 \in C$ be distinct.   Let $d:X\times X\rightarrow \mathbb{R}$ be any compatible metric.  Then, because the interval $[0,d(c_1,c_2))$ has the cardinality of $\mathbb{R}$, there must be an $r\in(0, d(c_1,c_2))$ for which $d(c_1,c)\neq r$ for any $c\in C$.  Then $U = \{c\in C: d(c,c_1)<r\}$ and $V = \{c\in C: d(p,c_1) > r\}$ disconnect $C$.

So, as $f(\N)$, is connected and countable, it must consist of a single point.

\end{proof}

\section{Continuous functions \texorpdfstring{$\N_{a_n}\rightarrow \N_{b_n}$}{Nan -> Nbn}}\label{sec:four}

We begin this section with the following observation about continuous maps $\N_{a_n}\rightarrow \N_{b_n}$.

\begin{proposition}\label{prop:NtoN}
Suppose $f:\N_{a_n}\rightarrow \N_{b_n}$ is continuous and $\sum_\N a_n $ diverges.  If $f$ is non-constant, then the image of $f$ is dense.  In particular, $f(\mathbb{N})$ must be an infinite subset of $\mathbb{N}$.

\end{proposition}

\begin{proof}This is obvious if $f$ is surjective, so we assume $f(\N)\subsetneq \N_{b_n}$.  By Theorem  \ref{thm:disc}, since $\sum_\N a_n$ diverges, $\N_{a_n}$ is connected.  Now, if $f(\N)\subseteq \N_{b_n}$ is closed, then in the subspace topology, the connected set $f(\N)$ is discrete, so must consist of a single point.  In particular, if $f$ is non-constant, $f(\N)$ cannot be closed.  But, we showed in Proposition \ref{prop:basicproperties} that in any series topology, a subset is either closed or dense.

\end{proof}

We now work towards providing partial results characterizing the homeomorphism type of $\N_{a_n}$.  Our first proposition provides some simple sufficient conditions which guarantee $\N_{a_n}$ is homeomorphic to $\N_{b_n}$.

\begin{proposition}\label{suffhomeo}  $\N_{a_n}$ and $\N_{b_n}$ are homeomorphic if any of the following conditions is satisfied.

\begin{enumerate}
\item  There is a bijection $\sigma:\N\rightarrow \N$ for which $a_n = b_{\sigma(n)}$ for all $n$.

\item  The limit $\lim_{n\rightarrow \infty} \frac{a_n}{b_n}$ exists and is positive.

\item  The space $\N_{b_n}$ is not cofinite, and $a_{2n-1} = b_n$, while $\sum_\N a_{2n}$ converges.

\end{enumerate} 

\end{proposition}

The last condition essentially says that a convergent series and a series $\sum_\N b_n$ with $\inf\{b_n\} = 0$ can be spliced together to give space homeomorphic to $\N_{b_n}$.  Using the first condition, it can be inserted anywhere and in any order.

\begin{proof}
For the first statement, $\sigma:\N_{a_n}\rightarrow \N_{b_n}$ is a homeomorphism.  In more detail, suppose $X\subseteq \N_{b_n}$ is proper.  Since  $$\sum_{\sigma^{-1}(X)} a_n = \sum_{\sigma^{-1}(X)} b_{\sigma(n)} = \sum_{X} b_n,$$ we see that $X\subseteq \N_{b_n}$ is closed iff $\sigma^{-1}(X)\subseteq \N_{a_n}$ is closed.

\

For the second statement, suppose $L = \lim_{n\rightarrow\infty} \frac{a_n}{b_n}$ with $0 < L < \infty$.  Then there is a natural number $N$ with the property that for any $n>N$, $$|\frac{a_n}{b_n} - L| < \frac{L}{2}.$$  In particular, $$\frac{L}{2}b_n < a_n < \frac{3L}{2}b_n$$ for all $n > N$.

We claim the identity function $\N_{a_n}\rightarrow \N_{b_n}$ is a homeomorphism.  Suppose $A\subseteq \N_{a_n}$ is proper and closed.  Decompose $A = A_{0}\cup A_{1}$ with $A_{0} = A\cap \{0,1,....,N\}$ and $A_1 = A\setminus A_0$.  Then $$\sum_{n\in A} b_n = \sum_{ A_0} b_n + \sum_{ A_1} b_n < \sum_{ A_0} b_n + \frac{2}{L}\sum_{ A_1} a_n.$$  The first sum contains only finitely many terms, so converges, and the second is bounded by $$\frac{2}{L}\sum_A a_n <\infty,$$ so $A$ is also closed in $\N_{b_n}$.  This shows the identity map is a closed map.  Interchanging $a_n$ and $b_n$ and using the fact that $a_n < \frac{3L}{2} b_n$ shows the identity is continuous, so it is a homeomorphism.

\

For the last statement, since $\N_{b_n}$ is not cofinite, $\inf\{b_n\} = 0$ (Theorem \ref{thm:cof}), so, by Lemma \ref{lem:closedsub} there is an infinite closed subset of $\N_{b_n}$.  By shrinking this subset, we may assume the complement is infinite.  Then, by (1) of this proposition, we may assume this subset is $E$, the even numbers.  Let $\psi:E\rightarrow D$, where $D=\{n\in \N: n\not\cong 1\pmod{4}\}$, denote any bijection.  Now, our homeomorphism $\sigma:\N_{b_n}\rightarrow \N_{a_n}$ is the map $$\sigma(n) = \begin{cases} 2n-1 & n\text{ odd}\\ \psi(n) & n\text{ even} \end{cases}.$$  To see that $\sigma$ is surjective, note that by the definition of $\psi$, we need only show that $2n-1$, $n$ odd, represents every number which is congruent to $1\pmod{4}$.  But $2n-1 = 4k+1$ is solved by the odd number $n = 2k+1$.  The fact that $\sigma$ is injective follows because both $n\mapsto 2n-1$ and $\psi$ are injective, together with the fact that $2n-1$ is always congruent to $1\pmod{4}$ if $n$ is odd.

We now show that $\sigma$ and $\sigma^{-1}$ are continuous, beginning with $\sigma$.  So, suppose $A\subseteq \N_{a_n}$ is proper and closed.  Decompose $A$ as $A = A_0 \cup A_1$ where $A_0 = A\cap D$ and $A_1 = A\setminus A_0$.  Then, $\sigma^{-1}(A_0) = \psi^{-1}(A_0)\subseteq E$.  By assumption, $\sum_E b_n < \infty$, so $\sum_{\sigma^{-1}(A_0)} b_n \leq \sum_E b_n < \infty$.

Also, we compute that $$\sum_{\sigma^{-1}(A_1)} b_n = \sum_{\sigma^{-1}(A_1)} a_{2n-1} = \sum_{A_1} a_n < \infty.$$  Thus, $$\sum_{\sigma^{-1}(A)} b_n = \sum_{\sigma^{-1}(A_0)} b_n + \sum_{\sigma^{-1}(A_1)} b_n < \infty.$$  This shows that $\sigma^{-1}(A)\subseteq \N_{b_n}$ is closed.

Finally, assume $A\subseteq \N_{b_n}$ is closed and proper.  We must show $\sigma(A)$ is closed.  Decompose $A = A_0\cup A_1$ where $A_0 = A\cap E$ and $A_1 = A\setminus A_0$.  Then $$\sum_{\sigma(A_0)} a_n = \sum_{\psi(A_0)} a_n \leq \sum_{D} a_n = \sum_{E} b_n + \sum_E a_n < \infty.$$

Also, $$\sum_{\sigma(A_1)} a_n = \sum_{A_1} a_{\sigma(n)} =\sum_{A_1} a_{2n-1} =\sum_{A_1} b_n \leq \sum_A b_n <  \infty.$$  Thus, $$\sum_{\sigma(A)} a_n = \sum_{\sigma(A_0)} a_n + \sum_{\sigma(A_1)} a_n < \infty.$$

\end{proof}

Theorem \ref{thm:hos} is a corollary of Proposition \ref{suffhomeo}(3).

\begin{proof}[Proof of Theorem \ref{thm:hos}]
Suppose $U\subseteq \N_{a_n}$ is a non-empty open set.  Then $U^c$ is a proper closed set.  If $U$ is finite, then $\sum_\N a_n = \sum_U a_n + \sum_{U^c} a_n < \infty$, a contradiction, so $U$ must be infinite.  List the elements of $U$ as $U = \{u_1, u_2 ,...\}$ with $u_1 < u_2 < ...$.  Consider the series $\sum_\N b_n$ where $b_n = a_{u_n}$.  Then $\sum_\N a_n$ is obtained from $\sum_\N b_n$ by adjoining the convergent series $\sum_{U^c} a_n$.  Proposition \ref{suffhomeo}(c) now implies that $U$ an $\N_{a_n}$ are homeomorphic.

\end{proof}

\subsection{A proof of Theorem \ref{thm:homeo}}

We now establish Theorem \ref{thm:homeo}, using a proof due to Alexey Lebedev \cite{MSE1}.  For the remainder of this section, we will assume $\lim_{n\rightarrow\infty} a_n = \lim_{n\rightarrow\infty} \frac{a_n}{b_n}=0$ and that $\sum_\N b_n$ diverges.  Notice in this case, that $\{a_n\}$ has a largest element, and a second largest element, etc.  Thus, by Proposition \ref{suffhomeo}(1), we assume without loss of generality that the sequence $\{a_n\}$ is non-increasing.  

We prove Theorem \ref{thm:homeo} via a sequence of lemmas.  We let $S_n = \sum_{i=1}^n a_i$ and $T_n = \sum_{j=1}^n b_j$ be the partial sums.

\begin{lemma}\label{lem:partialsum}We have $\lim_{n\rightarrow\infty} \frac{S_n}{T_n} = 0$.

\end{lemma}

\begin{proof}Because $\lim_{n\rightarrow \infty} \frac{a_n}{b_n} = 0$, for any $\epsilon > 0$, there is an $N$ with the property that $n > N$ implies $\frac{a_n}{b_n} < \epsilon$.  Then, for $n > N$, we have $$\frac{S_n}{T_n} = \frac{S_N + \sum_{N+1 \leq i \leq n} a_i}{T_n} < \frac{S_N + \epsilon \sum_{N+1\leq i \leq n} b_i}{T_n} = \frac{S_N + \epsilon\left(T_n - T_N\right)}{T_n}.$$

Thus, $\frac{S_n}{T_n} < \frac{S_N - \epsilon T_N}{T_n} + \epsilon$.  Now, $S_N$ and $T_N$ are constants, and $T_n\rightarrow \infty$ as $n\rightarrow \infty$, so clearly, we can make $\frac{S_n}{T_n}$ as small as we like by taking $n$ large enough.

\end{proof}

Now, suppose $f:\N\rightarrow \N$ is any bijection.  For any $\epsilon > 0$, we define $$M_\epsilon = \{i\in \N: a_{f(i)} < \epsilon b_i\}.$$

\begin{lemma}\label{lem:Meps}For any $\epsilon > 0$, $\sum_{M_\epsilon} b_i$ diverges.

\end{lemma}

\begin{proof}
Since $f$ is a bijection and $a_n$ is non-increasing, we have $\sum_{i\leq n} a_{f(i)}\leq \sum_{i\leq n} a_i = S_n$.

Thus, we compute $$\frac{S_n}{T_n} \geq \frac{\sum_{i\leq n} a_{f(i)}}{T_n} \geq \frac{\sum_{i\leq n, i\notin M_\epsilon} a_{f(i)}}{T_n} \geq \frac{\epsilon \sum_{i\leq n, i\notin M_\epsilon} b_i }{T_n} = \frac{\epsilon(T_n - \sum_{i\leq n, i\in M_\epsilon} b_i)}{T_n}.$$

Now, if $\sum_{M_\epsilon} b_i$ converges, then the fact that $T_n\rightarrow \infty$ as $n\rightarrow\infty$ implies that $\frac{S_n}{T_n} \geq \frac{\epsilon}{2}$ for $n$ large enough.  This contradicts Lemma \ref{lem:partialsum}.

\end{proof}

We can now finish the proof of Theorem \ref{thm:homeo}.

\begin{proof}

We now construct a set $A$ for which $\sum_{f(A)} a_n$ converges, but $\sum_{A} b_n$ diverges.  Assuming we can do this, this shows that $f$ cannot be a homeomorphism.

We construct $A$ as a disjoint union $A = \bigcup_m A_m$, where each $A_m$ is finite.  The sets $A_m$ are defined inductively, beginning with $A_0 = \emptyset$.

Now, assuming $A_0, A_1, ..., A_{m-1}$ have been defined and are finite, we now define $A_{m}$.  To do so, consider the set $$X_m:=M_{2^{-m}} \setminus \left(\left\{i\in \N: a_{f(i)} \geq 2^{-m}\right\} \cup \bigcup_{i<m} A_i \right).$$  Notice $\left\{i\in \N: a_{f(i)} \geq 2^{-m}\right\}$ is a finite set since $\lim_{n\rightarrow\infty} a_n = 0$.  Likewise, inductively, $\bigcup_{i<m} A_i$ is a finite set.  Since $X_m$ and $M_{2^{-m}}$ differ in only a finite set, Lemma \ref{lem:Meps} implies that $\sum_{X_m} b_n$ diverges.  In particular, there is a finite subset $Y_m\subseteq X_m$ for which $\sum_{Y_m} b_n > 1$.  Among all the subsets of $Z_m\subseteq Y_m$ for which $\sum_{Z_m} b_n > 1$, we let $A_m$ denote one with minimal cardinality.  Then $\sum_{A_m} b_n > 1$ but for any $x\in A_m$, $\sum_{A_m\setminus\{x\}} b_n \leq 1$.

Now, set $A = \bigcup A_m$.  We claim that $\sum_A b_n$ diverges.  Indeed, we have $$\sum_A b_n = \sum_{m=1}^\infty \sum_{A_m} b_n  \geq \sum_{m=1}^\infty 1.$$

We next claim that $\sum_{f(A)} a_n$ converges.  To see this, first, let $x_m$ be any element in $A_m$.  Since $x_m\in A_m\subseteq X_m$, $a_{f(x_m)} < 2^{-m}$.  Further, since $A_m \subseteq M_{2^{-m}}$ and $\sum_{A_m\setminus\{x_m\}} b_n \leq 1$, $$\sum_{A_m\setminus\{x_m\}} a_{f(n)} \leq 2^{-m} \sum_{A_m\setminus \{x_m\}} b_n \leq 2^{-m}.$$

Thus, we see that $$\sum_{f(A_m)} a_n = a_{f(x_m)} + \sum_{f(A_m\setminus \{x_m\})} a_n < 2^{-m} + \sum_{A_m\setminus\{x_m\}} a_{f(n)}  \leq  2^{-m} + 2^{-m} = 2^{-m+1}.$$  So, $$\sum_{f(A)} a_n =\sum_{m=1}^\infty \sum_{f(A_m)} a_n < \sum_{m=1}^\infty 2^{-m+1} < \infty.$$

\end{proof}

As a simple corollary, we see that for $0 < p < q \leq 1$, that $\N_{1/n^p}$ and $\N_{1/n^q}$ are not homeomorphic.  Thus, forming a series topology generates at least $|\mathbb{R}|$ many topologies which are distinct up to homeomorphism.  On the other hand, there are only $|\mathbb{R}|^{|\mathbb{N}|} = |2^{|\mathbb{N}|}|^{|\mathbb{N}|} = |2^{|\mathbb{N}||\mathbb{N}|}| = |\mathbb{R}|$ series.  So there are precisely $|\mathbb{R}|$-many topologies on $\N$ derived from series.

\subsection{A proof of Theorem \ref{thm:constant}}

We conclude with a proof of Theorem \ref{thm:constant}.  We will henceforth assume $p\in[0,1]$ and that $p < q$.  We must prove that any continuous map $f:\N_{1/n^p}\rightarrow \N_{1/n^q}$ is constant.

We handle the case $p = 0$ separately from the case $p > 0$.  For, if $p=0$, then $\N_{1/n^0}$ is the $\N$ with the cofinite topology, whereas $\N_{1/n^q}$ is not cofinite.  Since any non-constant continuous function $f:\N_{1/n^0}\rightarrow \N_{1/n^q}$ has infinite image (Proposition \ref{prop:NtoN}), Lemma \ref{lem:closedsub} gives an infinite closed $A\subseteq f(\N_{1/n^0})$.  Then $f^{-1}(A)$ is infinite, so not closed in $\N_{1/n^0}$.

We prove the case $p>0$ following an approach of Georgiy Shevchenko\cite{MSE2}.  The following lemma is due to Georgiy Shevchenko.

\begin{lemma}\label{lem:findA}
Suppose $a_n$ and $b_n$ are positive, that $\sum_\N b_n = \infty$, and that $\lim_{n\rightarrow\infty} b_n = \lim_{n\rightarrow\infty} \frac{a_n}{b_n} = 0$.  Then there is a subset $A\subseteq \N$ with the property that $\sum_A b_n$ diverges while $\sum_A a_n$ converges.

\end{lemma}

\begin{proof}
For each $k\in \N$, we let $N_k$ be a natural number with the property that $a_n/b_n\leq1/k$ for any $n \geq N_k$.  Because $\lim_{n\rightarrow \infty} b_n = 0$, but $\sum_\N b_n = \infty$, we can find $n_1', n_1 \geq N_1$ with $n_1'> n_1$ for which $$\sum_{i = n_1}^{n_1'} b_i \in (1,2).$$  Inductively, we can find $n_k', n_k \geq N_k$ with $n_k' > n_k > n_{k-1}'$ for which $$\sum_{i=n_k}^{n_k'} b_i \in\left(\frac{1}{k}, \frac{2}{k}\right).$$

Finally, we set $A = \bigcup_k\{n_k, n_k + 1, n_k + 2,..., n_k'\}$.  Because $$\sum_{\{n_k, n_k+1 ,..., n_k'\}} b_n > 1/k,$$ $\sum_A b_n > \sum_\N 1/n$, so diverges.  On the other hand, $$\sum_{\{n_k, n_k+1 ,..., n_k'\}} a_n \leq \frac{1}{k} \sum_{\{n_k, n_k+1 ,..., n_k'\}} b_n \in\left (\frac{1}{k^2}, \frac{2}{k^2}\right),$$ so $\sum_A a_n \leq 2\sum_\N  1/n^2 < \infty$.

\end{proof}

Now, suppose $0< p\leq 1$ and $p< q$ and that $f:\N_{1/n^p}\rightarrow \N_{1/n^q}$ is any non-constant continuous function.  For each $i\in \N$, we let $u_i = \sum_{f^{-1}(i)} 1/n^p$, with the convention that this sum is zero if $f^{-1}(i)$ is empty.  Note $\sum_{f^{-1}(i)} 1/n^p$ is finite, as $f^{-1}(i)$ is a closed subset of $\N_{1/n^p}$.

\begin{proposition}\label{prop:unconv}  If $f$ is continuous, then $\lim_{n\rightarrow \infty} u_n = 0$.

\end{proposition}

\begin{proof}

Assume not.  Then there is some $\epsilon > 0$ with the property that $u_n \geq \epsilon$ for $n$ in some infinite set $A\subseteq \N$.  From Lemma \ref{lem:closedsub}, we can find an infinite subset $B\subseteq A$ for which $\sum_B 1/n^q<\infty$.  Then $B$ is closed in $\N_{1/n^q}$, but $$\sum_{f^{-1}(B)} \frac{1}{n^p} = \sum_B u_n \geq \sum_\N \epsilon,$$ which diverges.  This contradicts the fact that $f$ is continuous.

\end{proof}

Now, we claim that $\sum_\N u_n^{1/p}$ diverges.  Indeed, since $1/p \geq 1$, the function $x\mapsto x^{1/p}$ is convex, and thus, super-additive.  In particular, for each term $$u_i^{1/p} = \left(\sum_{f^{-1}(i)} \frac{1}{n^p}\right)^{1/p} \geq \sum_{f^{-1}(i)} \frac{1}{n}.$$  Then $$\sum_\N u_i^{1/p} \geq \sum_\N \sum_{n\in f^{-1}(i)} \frac{1}{n} = \sum_\N \frac{1}{n},$$ so diverges.

Now, let $r\in(1, q/p)$.  Then clearly $$\sum_{\{n:u_n^{1/p} \leq \frac{1}{n^r}\}} u_n^{1/p} \leq \sum_\N \frac{1}{n^r} < \infty,$$ so it follows that $$\sum_{\{n:u_n^{1/p} > \frac{1}{n^r}\}} u_n^{1/p}$$ diverges.  We let $$C =\{n\in \N: u_n^{1/p} > \frac{1}{n^r}\}$$ and write $C = \{c_1, c_2, ...\}$ where $c_1 < c_2 $, etc.  Thus $\sum_\N u_{c_n}^{1/p}$ diverges.  Since $1/p \geq 1$ and $u_n\rightarrow 0$ as $n\rightarrow \infty$, it follows that for large $n$, $u_{c_n} \geq u_{c_n}^{1/p}$.  In particular, we have the following proposition.

\begin{proposition}\label{prop:sumdiv}
The series $\sum_\N u_{c_n}$ diverges.

\end{proposition}

Finally, we have the following proposition.

\begin{proposition}\label{prop:lim}  We have $$\lim_{n\rightarrow\infty} \frac{1/(c_n)^q}{u_{c_n}} = 0.$$

\end{proposition}

\begin{proof}  Since $c_n \in C$, each $u_{c_n} > 0$, so the terms in the limit are defined.  Now, note that if $u_n^{1/p} > \frac{1}{n^r}$, then $u_n > \frac{1}{n^{rp}}$.  For large $n$, we see $$0\leq \frac{1/(c_n)^q}{u_{c_n}} \leq  \frac{1/(c_n)^q}{1/(c_n)^{rp}}= c_n^{rp-q}.$$  Since $r < q/p$, we have $rp -q < 0$.  Further, $c_n\rightarrow \infty$ as $n\rightarrow \infty$.  Thus $\lim_{n\rightarrow\infty} c_n^{rp-q} = 0$.  The result now follows from the squeeze theorem.
\end{proof}

If $f$ is continuous, then Propositions \ref{prop:unconv}, \ref{prop:sumdiv}, and \ref{prop:lim} exactly show that that the sequences $a_n = 1/(c_n)^q$ and $b_n = u_{c_n}$ satisfy the conditions of Lemma \ref{lem:findA}.  In particular, there is a set $A\subseteq \N$ for which $\sum_{A} u_{c_n}$ diverges, but $\sum_A \frac{1}{(c_n)^q}$ converges.  This, in turn, means there is a subset $B\subseteq \N$ for which $\sum_{B} u_n$ diverges, but $\sum_{B} \frac{1}{n^q}$ converges.  However, we have $$\sum_{f^{-1}(B)} \frac{1}{n^p} = \sum_{i\in B} \sum_{n\in f^{-1}(i)} \frac{1}{n^p} = \sum_{B} u_n = \infty.$$  Thus, $B$ is closed in $\N_{1/n^q}$, but $f^{-1}(B)$ is not closed in $\N_{1/n^p}$.  This contradicts the fact that $f$ is continuous, and concluded the proof of Theorem \ref{thm:constant}.

\bibliographystyle{plain}
\bibliography{zachbib}

\begin{thebibliography}{1}

\bibitem{Fl}
Jana Fla\v{s}kov\'{a}.
\newblock Description of some ultrafilters via i-ultrafilters.
\newblock {\em Proc. RIMS}, 1619, 2008.

\bibitem{FrRa}
S.~P. Franklin and M.~Rajagopalan.
\newblock Spaces of diversity one.
\newblock {\em J. Ramanujan Math. Soc.}, 5:7--31, 1990.

\bibitem{Hr}
Michael Hru\v{s}\'{a}k.
\newblock Combinatorics of filters and ideals.
\newblock {\em Contemporary Mathematics}, 533:29--69, 2011.

\bibitem{MSE1}
Litho (Alexey~Lebedev) (https://math.stackexchange.com/users/197288/litho).
\newblock Is there a bijection of the natural numbers which swaps
  $\frac{1}{n}$-summable subsets with $\frac{1}{\sqrt{n}}$-summable subsets?
\newblock Mathematics Stack Exchange.
\newblock URL:https://math.stackexchange.com/q/2429818 (version: 2017-09-18).

\bibitem{Ka}
Miroslav Kat\v{e}tov.
\newblock Products of filters.
\newblock {\em Commentationes Mathematicae Universitatis Carolinae},
  9(1):173--189, 1968.

\bibitem{KwTr}
A.~Kwela and J.~Tryba.
\newblock Homogeneous ideals on countable sets.
\newblock {\em Acta Mathematica Hungarica}, 151(1):139--161, Feb 2017.

\bibitem{Mr}
N.~Mro\.{z}ek.
\newblock Some applications of the {K}at\v{e}tov order on {B}orel ideals.
\newblock {\em Bull. Pol. Acad. Sci. Math.}, 64:21--28, 2016.

\bibitem{Sier}
W.~Sierpi{\'n}ski.
\newblock Un th{\'e}or{\`e}me sur les continus.
\newblock {\em T{\^o}hoku Mathematical Journal}, 13:300--303, 1918.

\bibitem{MSE2}
zhoraster (Georgiy
  Shevchenko)~(https://math.stackexchange.com/users/262269/zhoraster).
\newblock Is there a bijection of the natural numbers which swaps
  $\frac{1}{n}$-summable subsets with $\frac{1}{\sqrt{n}}$-summable subsets?
\newblock Mathematics Stack Exchange.
\newblock URL:https://math.stackexchange.com/q/2434530 (version: 2017-09-18).

\end{thebibliography}

\end{document}